\documentclass[a4paper,12pt]{amsart}
\usepackage{amsmath,amssymb}
\def\C{\Bbb C}
\def\D{\Delta}

\def\k{\kappa}
\def\vp{\varphi}

\newcommand\DD{{\mathbb D}}

\newcommand\BB{{\mathbb B}}

\def\im{\operatorname{Im}}
\def\re{\operatorname{Re}}
\def\diam{\operatorname{diam}}
\def\dist{\operatorname{dist}}

\newtheorem{theorem}{Theorem}
\newtheorem{lemma}[theorem]{Lemma}

\newtheorem{definition}[theorem]{Definition}
\newtheorem{proposition}[theorem]{Proposition}
\newtheorem{corollary}[theorem]{Corollary}
\theoremstyle{remark}\newtheorem{remark}[theorem]{Remark}

\title[Gehring-Hayman for strongly pseudoconvex domains]{A Gehring-Hayman inequality for~strongly pseudoconvex domains}

\author[\L ukasz Kosi\'nski]{\L ukasz Kosi\'nski}
\address{\L. Kosi\'nski\\ Institute of Mathematics, Faculty of Mathematics and Computer Science, Jagiellonian University, \L ojasiewicza 6, 30-348 Krak\'ow, Poland} \email{lukasz.kosinski@uj.edu.pl}

\author{Nikolai Nikolov} 
\address{N. Nikolov\\Institute of Mathematics and Informatics\\Bulgarian Academy	of Sciences\\Acad. G. Bonchev 8, 1113 Sofia, Bulgaria \vspace{1mm}	\newline Faculty of Information Sciences\\	State University of Library Studies and Information Technologies\\Shipchenski prohod 69A, 1574 Sofia, Bulgaria}

\email{nik@math.bas.bg}

\author{Pascal J. Thomas}
\address{P.J.~Thomas\\
Institut de Math\'ematiques de Toulouse; UMR5219 \\
Universit\'e de Toulouse; CNRS \\
UPS, F-31062 Toulouse Cedex 9, France} 

\email{pascal.thomas@math.univ-toulouse.fr}

\thanks{The first named author was supported by the NCN grant SONATA BIS no. 2017/26/E/ST1/00723. The second named author was partially supported by the Bulgarian National Science Fund, Ministry of Education and Science of Bulgaria under contract KP-06-N52/3. The second and third author enjoy the 
support of the PHC Rila program 48135TJ, which made possible the stay at
the Institute for Mathematics and Informatics of the Bulgarian Academy of Sciences,
Sofia, during which this work was started.}


\begin{document}
	
	
\begin{abstract}
We prove that if $D$ is a strongly pseudoconvex domain with $\mathcal C^{2, \alpha}$-smooth boundary, then the length of a geodesic for the Kobayashi-Royden infinitesimal metric between two points is bounded by a constant multiple of the Euclidean distance between the points.
\end{abstract}

	\maketitle

\section{Introduction and motivation}

In 1962, F.W. Gehring and W.K. Hayman \cite{GH} proved a conjecture of G. Piranian: 
there is an absolute constant $C$ such that for   any
simply connected domain $U$ strictly included in the complex plane, for any $z_1, z_2 \in U$,
if $\gamma_H$ is the hyperbolic geodesic from $z_1$ to $z_2$, and $\gamma$ any
rectifiable curve from $z_1$ to $z_2$ contained in $U$, then
\[
l\left( \gamma_H \right) \le C l\left( \gamma\right) ,
\]
where $l$ denotes Euclidean length. When $U$ is convex, this just means that
$l\left( \gamma_H \right) \le C |z-w|$. There are also versions of this 
result for the diameters of the curves and for other metrics, see for instance the survey \cite{PS}.

In the several complex variable settings, it is natural to ask similar questions
about geodesics for the Kobayashi-Royden metric
$$\k_D(z;X)=\inf\{|\alpha|:\exists\vp\in\mathcal O(\D,D), \vp(0)=z, \alpha\vp'(0)=X\},$$
 and it seems
realistic to restrict oneself to domains with some regularity. 
The best bounds known so far were of the form $C|z-w|^{1/m}$ for 
$m$-convex domains and strongly pseudoconvex domains, see \cite{LWZ}, \cite{LPW} and \cite{NO}.
However, in the case of the unit ball, it is easy to see that geodesics,
since they lie inside affine complex lines, must have length bounded by 
$\frac\pi2|z-w|$, which is much smaller in the critical case when $z$ and $w$
are close to each other.  The goal of this note is to show that this type of bound
still holds in regular enough strongly pseudoconvex domains.

	\begin{theorem}\label{thm:main}
Let $D$ be a strongly pseudoconvex domain with $\mathcal C^{2, \alpha}$-smooth boundary. Then there is $C>0$ such that $$|z-w|\leq l(\gamma_{z,w}) \leq C |z-w|$$
for any $z,w\in D$ and any Kobayashi geodesic $\gamma_{z,w}$ joining $z$ and $w$.
\end{theorem}

Observe that $C$ here depends on the domain. For this result, it does not matter whether
we consider the lengths of all curves or simply the Euclidean distance: since the
domain is bounded, we may always assume that $|z-w|\le c_0$ for come $c_0>0$,
the inequality being trivial otherwise, and that $z,w$ are close to $\partial D$,
since when $z,w$ are contained in some compact set $K\subset\subset D$, 
the Kobayashi-Royden and Euclidean metrics are equivalent, and so
$ l(\gamma_{z,w})$ is comparable to $|z-w|$. And for two points close to each other
and to a $\mathcal C^2$ boundary, the lengths of curves connecting the points are
comparable to their Euclidean distance.

The optimal class of domains to which this result can apply remains an open question: can we
lower the regularity requirement? Can we allow finite type or even 
mere pseudoconvexity? Is that property even required? 
For really irregular domains,
do we have a bound involving the lengths
of connecting curves rather than the Euclidean distance?
Can we have a universal constant if the bound involves the lengths
of connecting curves?
Of course, geodesics
for the Kobayashi-Royden metric do not always exist, 
and the question would have to be rephrased using the existing notions of quasi-geodesics.

\section{A construction of Fridman-Ma and Lempert's scaling}\label{sec:FM}

Within this section, we shall outline the \emph{global scaling} process which provides much more precise estimates than standard methods used so far (e.g. the ones involving squeezing functions).

In his seminal paper \cite{L1}, Lempert showed 
that in a strongly 
convex domain $D$
with $\mathcal C^6$-boundaries any two points lie on a 
complex geodesic disc.
To show the existence of a geodesic he 
used the scaling method: for a point $p$ in the topological boundary $\partial D$ of a strongly pseudoconvex domain $D$ he found a neighborhood $U$ and a biholomorphic map $F$ on $U$ such that $F(\partial D\cap U)$ was close in $\mathcal C^6$-topology to a big part of the Euclidean sphere. The main weakness of the above construction is that it is not defined globally on $D$.

A similar construction was independently applied by Fridman and Ma \cite{FM} to the problem of exhaustion of smooth pseudoconvex domains at points of global strict convexity. 

We will construct a global scaling on strongly pseudoconvex domains. \emph{Scaling} here means that we use a precise automorphism of the Euclidean ball to scale while \emph{global} means that scaling maps are defined on the whole $\overline D$, so that images of geodesics under scaling maps are still geodesics, while an explicit form for automorphisms is essential in carrying out precise estimates.

We shall try to keep the notation from both papers: $\mathbb B$ denotes the unit Euclidean ball and $B$ denotes the ball of radius $1$ centered at $(-1,0)$, that is $B= \mathbb B- (1,0)$.

\bigskip

{\bf Fridman-Ma construction}. Let $D$ be a strongly pseudoconvex domain with $\mathcal C^{2,\alpha}$-smooth boundary, $\alpha\in [0,1]$. Note that for $\alpha=0$ this covers 
the $\mathcal C^2$ case. Thanks to \cite{DFW}, that any boundary point  $p\in \partial D$ can be exposed, that is to say, there is a  biholomorphism $\Phi$ on a neighborhood of $\overline D$ such that for $D'=\Phi(D)$, $p'=\Phi(p)$, $p'$ is a point of \emph{ global strong convexity}, that is, there is a ball $U$ such that $D\subset U$ and $p' \in \partial D \cap \partial U$. This will be the starting point for the Fridman-Ma construction  \cite{FM}, which we shall recall, but in addition we will keep track of the dependence of the transformations on the starting point $p$.


Let us start with the following definition.
\begin{definition}[See {\cite[Definition 3.1]{FM}}]
Let $N,k$ be natural numbers, $N\geq 4$, $k\leq  2(N-4)$. A $\mathcal C^2$ smooth domain $D\subset \mathbb C^n$, $0\in \partial D$ is of type $C(N,k)$ if there is a defining function $\rho$ of $D$ near $0$ of the form 
\begin{equation*}
\rho(z) = 2 \re z_1 + \re \sum_{i,j=2}^n a_{ij} z_i z_j + d |z_1|^2  + N|z'|^2 + o(|z|^{2+\alpha}),
\end{equation*}
where $z'=(z_2,\dots,z_n)$, the numbers $N, d$ and $a_{ij}$ satisfy $d>1$ and 
\begin{equation}\label{eq:aij}
\re \sum_{i,j=2}^n a_{ij} z_i z_j + (N-4 - k/2) |z'|^2 \geq 0 \quad \text{for } z'\in\mathbb C^{n-1}.
\end{equation}
\end{definition}
\begin{remark}\label{CNK}Note that \eqref{eq:aij} is equivalent to $|\sum_{i,j=2}^n a_{ij} z_i z_j| \leq (N-4 - k/2) |z'|^2$. In particular, if $k= 2(N-4)$, then $a_{ij}=0$.
\end{remark}

Since $p$ is a point of global strong convexity, the same is true of nearby points. The construction is outlined in five steps below.

\medskip

{\it Step 1}. The translation 
\begin{enumerate}
\item[$i)$] $z\mapsto T_p(z):=z-p$,
\end{enumerate} 
the scaling (locally independent of $p$) and the 
unitary transformation
\begin{enumerate}
\item[$ii)$] $z \mapsto A(p)z$,
\end{enumerate}
which can be chosen so that if we write $\Phi_1:=A(p)\circ T_p,$
$\Phi_1(D)\subset B$ and it admits
a defining function $\rho$ 
with expansion at 
$0$ of the form  
\begin{multline*}\rho(z) = 2\re z_1 + \re \sum_{i,j=1}^n a_{ij}(p)z_i z_j +\\ \re \sum_{j=1}^n c_j(p) z_1 \bar z_j + \sum_{j=2}^n N_j(p)|z_j|^2 +  + o(|z|^{2+(1-\epsilon)\alpha}),
\forall \epsilon >0,
\end{multline*}
where we also have used $A(p)$ to diagonalize the possible hermitian terms
$\sum_{i,j=2}^n c_{i,j} z_i\bar z_j$.
The unitary matrix depends on the first and second-order derivatives of the defining function of $D$ at $p$, so $a_{ij}$, $c_j$ and $N_j$ are functions of class $\mathcal C^\alpha$ of $p$. 

\medskip

{\it Step 2}. This is precisely \cite[Lemma 3.2]{FM}. Namely, there is an integer $N\geq 4$ such that $D$ is biholomorphic to a domain of the type $C(N,0)$.

This is achieved by Fridman and Ma using an affine map  
\begin{enumerate} 
\item[$iii)$] $\Phi_2(z_1, z_2, \dots, z_n)=
\left( z_1, z_2+\frac{c_2(p)}{2N_2(p)} z_1, \dots, 
z_n+\frac{c_n(p)}{2N_n(p)} z_1\right)$ 
\end{enumerate}
the coefficients of
which are clearly of class $\mathcal C^\alpha$, as are their inverses; rescaling maps $\Phi_3$
\begin{enumerate} 
\item[$iv)$] $z_j\mapsto t_j(p) z_j$,
\end{enumerate}
chosen so that $\Phi_3 \circ \Phi_2 \circ \Phi_1 (D)\subset B$;
 the $t_j$ are of class $\mathcal C^\alpha$ with respect to $p$. Then, as in \cite[(3.7), p. 391]{FM} $\Phi_3 \circ \Phi_2 \circ \Phi_1 (D)$ admits a defining function of the form
\begin{multline}
\label{rho37}
  \rho(z) = 2\re z_1 + 2  \re \sum_{j=1}^n b_{j}(p)z_1 z_j 
+\\ 
\re \sum_{i,j=1}^n a_{ij}(p)z_i z_j + d|z_1|^2
+M |z'|^2+  + o(|z|^{2+(1-\epsilon)\alpha}),
\forall \epsilon >0.  
\end{multline}
We then apply an automorphism $\Phi_4=\Phi_4^\epsilon$ of $B$ of the form
\begin{enumerate}
\item[$v)$] $\Phi_4^\epsilon: z\mapsto \left( \frac{\epsilon z_1}{ 2- \epsilon + (1-\epsilon) z_1}, \frac{\sqrt{\epsilon (2-\epsilon)}}{2- \epsilon + (1-\epsilon) z_1} z'\right).$
\end{enumerate}
where $\epsilon$ is such that $\epsilon \sum_j |b_j(p)|$ is smaller than a given positive  number $\lambda$ chosen as in 
\cite[Lemma 3.1, p. 390]{FM}
($\lambda=1/8$ will do). In particular, since 
the $b_j$ are continuous w.r.t to $p$, $\epsilon$ can be chosen locally independently of $p$. The resulting domain
$\Phi_4 \circ \cdots \Phi_1(D)$ admits a defining function of the 
form \eqref{rho37} with new coefficients $b'_j$ satisfying the
hypotheses of \cite[Lemma 3.1]{FM} and so applying a map
$\Phi_5(z)=(z_1+\sum_{j=1}^n b'_j z_1 z_j,z')$, then
a contraction $\Phi_6(z)= \frac{M}N z$ with $N$ an integer,
$N>8M$ we obtain that $\Phi_6 \circ \cdots \Phi_1(D)$
is of type $C(N,0)$ \cite[end of proof of Lemma 3.2, p. 392]{FM}.
Again, it is straightforward to see that $N$ can be chosen locally constant.

\medskip

{\it Step 3}. The goal is to show that any domain of type $C(N,k)$ is biholomorphic to $C(N,k+1)$. This is relevant
only when $N\geq 5$ and $k\leq 2(N-4)-1$. Then the statement is precisely \cite[Lemma 3.4]{FM}. 
The transformations used to achieve this statement are:
\begin{enumerate}
\item[$vi)$] $z\mapsto  (z_1 + \sum_{i.j=2}^n b_{ij}(p) z_i z_j, z')$,
\end{enumerate}
where $b_{ij}(p) = \frac{1}{4(N-4-k/2)} a_{ij}(p)$,
\begin{enumerate}
\item[$vii)$] an automorphism $\Phi_4^\epsilon$ of the ball $B$,
\end{enumerate}
where $\epsilon$ is chosen independently of $p$, and
\begin{enumerate}
\item[$viii)$] $z\mapsto \left( \frac{14}{15} z_1, \sqrt{\frac{14}{15}} z'\right)$.
Call $\Phi_7$ the composition of those maps.
\end{enumerate}

\medskip

{\it Step 4}. Repeating Step~3 exactly $2(N-4)$ times, 
with maps $\Phi_7,\dots,\Phi_{2N-2}$,
we get a domain whose defining function near $0$ is of the form
\begin{equation}\label{eq:rho1}\rho(z) = 2 \re z_1 + d(p) |z_1|^2 + N|z'|^2 +  + o(|z|^{2+(1-\epsilon)\alpha}),
\forall \epsilon >0.
\end{equation}

\medskip

{\it Step 5}. Applying a map 
\begin{enumerate}\item[$ix)$] 
$\Phi_{2N-1}(z)= (d(p)z_1, (d(p)N)^{1/2} z'),$
\end{enumerate}
we can assume that $d(p)=N=1$. Also, the coefficients are of class $\mathcal C^\alpha$. Thus $\Phi_{2N-1}\circ \cdots \Phi_1(D)$
which admits a defining function  of the form 
\begin{equation}\label{rho}\rho(z) =2 \re z_1 + |z|^2 +  + o(|z|^{2+(1-\epsilon)\alpha}),
\forall \epsilon >0,\quad \text{as } z\to 0,
\end{equation}
and that is contained in the ball $rB$ of radius $r$ centered at $(-r,0)$ for some $r\geq 1$.

\begin{definition}
\label{fmmap}
If $D$ is a $\mathcal C^2$ domain
and $p\in \partial D$ a point of strong pseudoconvexity,
we denote  $F_p := \Phi_{2N-1}\circ \cdots \Phi_1$ the map
which results from the Fridman-Ma construction. 
\end{definition}
\medskip

\begin{remark}\label{remhol}
It follows from the above construction that we can locally write $F_q(z)=\tilde F(z, \alpha_1(q),\ldots, \alpha_M(q))$, where $\tilde F$ is a holomorphic mapping and $\alpha_j$ are of class $\mathcal C^{\alpha}$ (here $M$ denotes the number of $\mathcal C^\alpha$ smooth functions that appear in the construction of $F_q$). The same remains true for $F_q^{-1}(z)$.

In particular, the Jacobian of $F_q$ with respect to $z$ as well as 
that of its inverse $F_q^{-1}$ vary locally $\mathcal C^\alpha$-continuously. In particular, $F_q$ locally does not change the Euclidean distance between points.
\end{remark}

For later purposes, we need to know that we can choose an appropriate $p$
to make the Fridman-Ma construction.

\begin{lemma}
\label{rem} For any $z\in D$ that is sufficiently close to $p_0\in \partial D$,
there exists $q\in \partial D$, close to $p_0$, so that $F_q(z)\in (-1,0)\times \{0\}^{n-1}.$ 
\end{lemma}
\begin{proof}
Applying one Fridman-Ma transformation we can assume that $p_0=0$ and a defining $\rho$ function of $D$ near $p_0$ is of the form $\rho(z) = 2 \re z_1 +|z|^2 + + o(|z|^{2+(1-\epsilon)\alpha}),$ as as $z\to 0$.

As in Remark~\ref{remhol} write $F_q(z)=\tilde F(z, \alpha_1(q), \ldots, \alpha_M(q))$, where $\alpha_j$ are $\mathcal C^\alpha$-continuous, $\alpha_j(0)=0$, $j=1,\ldots, M$. 

Let us consider the map
$$\varphi:(-1,\delta)\times \partial D\ni (t,q)\mapsto F_q^{-1}(t,0)\in \mathbb C^n,$$ which is well defined for $\delta>0$ sufficiently small.  In particular,
\begin{equation*}\varphi(t,q) = G(t, \alpha_1(q), \ldots, \alpha_M(q))
\end{equation*}
for $q\in \partial D$ close $0$, where $G$ is holomorphic near $0$ in $\mathbb C^{1+M}$. Moreover,
\begin{equation*}\label{Fq1} \varphi(t,0)=(r t+o(t^2),0) \ \text{for some $r>0$, } \ \text{and } \varphi(0,q)=q
\end{equation*}
(the first equality above follows from formulas $(iii)-(ix)$).
Consequently, we can write 
\begin{equation}\label{Fq}\varphi(t,q) = (rt,0) + q + t \Gamma(t,\alpha_1(q),\ldots, \alpha_M(q)),
\end{equation} for $\Gamma$ holomorphic near $0$, $\Gamma(0)=0$.

We need to show that there is $\epsilon>0$ and a neighborhood $U$ of $p$ in $\partial D$ such the range $\varphi((-\epsilon, \epsilon)\times U)$ contains a neighborhood of $0$. This is already claimed in the proof of \cite[Theorem 4.1]{DGZ}. The authors deduced this fact by claiming that $\varphi$ is a local diffeomorphism near $t=0$ and $q=0$. This, however, is not the case, if $\alpha<1$, as the mapping $\varphi$ is only $\mathcal C^\alpha$-smooth with respect to the variable $q$.

Nevertheless, a topological argument proves the above assertion also in the case when $\alpha=0$, i.e. when $\partial D$ is $\mathcal C^2$ smooth (in particular \cite[Theorem 4.1]{DGZ} is salvaged in its full generality). Let us present how it can be done.

Let us write $\partial D$ near $0$ as $\re z_1 =\tilde \rho(\re z_2, z')$, where $\tilde \rho$ is $\mathcal C^{2, \alpha}$ smooth. For  simplicity of  notation, we shall identify $\mathbb C^n$ with $\mathbb R^{2n}$ writing $(z_1,\ldots, z_n)= (s, x)=(s, x_2, \ldots, x_{2n})$. Let $V$ be a neighborhood of $0$ in $\mathbb R^{2n-1}$ and put $\psi(s,x) = (s + \tilde\rho(x),x)$ if $x\in V$ so that $\psi$ is a $\mathcal C^2$-smooth map that sends $\{0\}\times V$ diffeomorphically to $U$.

Consider $\Phi: (-\epsilon,\epsilon)\times V \to \mathbb R^{2n}$ given by $\Phi(t,x)=\psi^{-1}( \varphi (t, \psi(0,x))).$ Our aim is to show that $\Phi((-\epsilon,\epsilon)\times V)$ is a neighborhood of $0$. In other words, we want to show $(s,y_1,\ldots, y_{2n})$ is in the image of $\Phi$, i.e. that there is $(t,x)$ such that 
\begin{equation}\label{fix}\Phi(t,x)= (s,y)
\end{equation}
providing that $s$ and $y$ are small.
Write $\Phi=(\Phi_1, \ldots, \Phi_{2n})$. 

It follows from \eqref{Fq} that
$$\Phi(t,x) = (rt, x) + t \gamma(t,x),$$ where $\gamma$ is continuous and $\gamma(0,0)=0$.


For a fixed $x$ we are looking at the equation $\Phi_1(t,x) = s$. It follows from the implicit function theorem ($\Phi_1$ depends smoothly on $t$ as well as on continuous functions $\alpha_j$ that we treat as coefficients) that this equation has a solution $t=t(s,x)$ which is of the form 
\begin{equation}\label{tfix}t= s/r + s\delta(s,x)
\end{equation}where $\delta$ is continuous, $\delta(0)=0$.


Plugging $t=s/r + s\delta(s,x)$ to $(\Phi_2, \ldots, \Phi_n)(t,x)$ we get the mapping of the form $$ x\mapsto x + s \beta(s,x),$$ where again $\beta$ is continuous and $\beta(0)=0$. Fix a closed ball in $V$ with $0$ being its center and denote it by $K$. Take $s$ sufficiently small so that $s \beta(s,x)$ is in $\frac12 K$ for any $x\in K$. Then, by the Brouwer fixed point theorem, $\frac12 K$ is contained in the image of $x\mapsto x + s \beta(s,x)$ (this is because for any $y\in \frac12 K$ the mapping $x\mapsto y- s\beta(s,x)$ sends $K$ to $K$, so must have a fixed point). Thus, if $s$ is small and $y\in \frac12 K$, equation \eqref{fix} has a solution $x$; then $t$ is given by \eqref{tfix}.


\end{proof}

{\bf Lempert's scaling method}. 

We will use automorphisms of the ball $\mathbb B$
\begin{equation}\label{eq:aut}A_t(z) = \left(m_t(z_1), \sqrt{1-t^2} \frac{z'}{1 + tz_1}\right),\text{ where } m_t(z_1)=\frac{z_1+ t}{1+ tz_1}.
\end{equation} 
Given a domain of the form $F_p(D)$, with 
$\partial D$ of class $\mathcal C^{2,\alpha}$ and $F_p$ as in
Definition \ref{fmmap}, we consider $D_0=T_{e_1}(D)$
where $e_1=(1,0)\in \C \times \C^{n-1}$ and $T_{e_1}$ is 
the translation $z\mapsto z+e_1$.
Then $D_0$ admits a defining function  of the form  $\rho_0(z) = -1 +|z|^2 +  + o(|z|^{2+(1-\epsilon)\alpha}),
\forall \epsilon >0,$ as $z\to e_1,$
and that is contained in the ball $B'$ of radius $1+s'$ centered at $(-s',0)$ for some $s'\geq 0$. Note that $A_t^{-1}$ is well defined on $B'$ and $A_t^{-1}(z)\to (-1,0)$ uniformly on compact subsets of $B'\setminus \{e_1\}$ as $t\to 1$. Define 
\begin{equation}\label{eq:Dt} D_t=A_t^{-1}(D_0)\quad \text{and} \quad \rho_t(z) = \frac{|1 + t z_1|^2}{1- t^2} \rho_0(A_t(z)),\ t\in (0,1).
\end{equation}
Clearly $D_t=\{\rho_t<0\}$ and $\rho_t(z)$ converges to $-1 +|z|^2$ locally uniformly in $\{\re z_1>-\alpha\}$ for any $\alpha<1$ as $t\to 1$. 

Summing up, for any $\beta>-1$:
\begin{enumerate}
\item[A1)] $D_t\cap \{\re z_1>\beta\}$ converges to $\mathbb B \cap \{\re z_1>\beta\}$ in $\mathcal C^{2,\alpha}$-topology (here and in the sequel this means convergence of defining functions),
\item[A2)] $D_t$ converges to the ball in the Hausdorff topology.
\end{enumerate}



\medskip


We now must
restrict ourselves to the case when $\partial D$ is of class $\mathcal C^{2,\alpha}$ with
$\alpha>0$. The reason is that we must use Lempert's work on complex geodesics \cite{L1, L3, L4}
(see also a discussion in \cite[Remark B]{H3})
and the arguments about the dependence and regularity of complex geodesics 
 fail if we merely assume $\mathcal C^2$-smoothness \cite[remark after the Main Theorem, p. 561]{L3}. 

In \cite{L3} and \cite{L4} Lempert considered the class $S^k$, where $k>1$ is non-integer; we refer the reader to \cite[p. 561]{L3} for a precise definition. If $D\in S^{m+\alpha}$, $m\in\mathbb N$, $0<\alpha<1$,
then $\partial D$ is of class $\mathcal C^{m,\alpha}$;
if $\partial D$ is of class $\mathcal C^{m,\alpha}$, then  $D\in S^{m-1+\alpha}$. 
Thus it is proven in \cite[Corollary 3.3]{L3} that any complex geodesic in a $\mathcal C^{2,\alpha}$-smoothly bounded strictly convex domain $D$ is in $\mathcal C^{1,\alpha}(\overline{\mathbb D})$. 
Also, the $\mathcal C^{1/2}(\overline{\mathbb D})$ norm of a complex geodesic $f$ is uniformly bounded when $f(0)$ is within a compact subset of $K$ of $D$, with bounds depending only on the diameter of $D$ and
normal curvatures of $\partial D$ and on $K$ \cite[Proposition 13]{L1}, so there is local uniformity with respect to the domain as well, in the sense of $\mathcal C^{2,\alpha}$ convergence of defining functions. 

Using those a priori estimates in the $\mathcal C^6$ case and approximation (see \cite[p. 562--563]{L3}), Lempert showed in \cite[Proposition 3]{L4} that the bound on $\mathcal C^{1,\alpha}(\overline{\mathbb D})$ norm is in fact locally uniform. 
This implies the following. 

\begin{proposition}
\label{prop:c1conv}
Let $D_n$ be a sequence of strictly convex domains such that their defining functions converge 
in the $\mathcal C^{2,\alpha}$ sense to the defining function of 
a strictly convex domain $D$. Let $K$ be a compact subset of $D$ and $f_n$  a sequence of complex geodesics in $D_n$ such that $f_n(0)\in K$. Then $\mathcal C^{1,\alpha}(\overline{\mathbb D})$-norms of $f_n$ are uniformly bounded, for $n$ big enough, and a subsequence of $(f_n)$ converges in the $\mathcal C^1(\overline{\mathbb D})$ topology
to a complex geodesic we denote by $f$.
\end{proposition}
\begin{proof}
 It follows from the hypothesis that the $\mathcal C^{\alpha}(\overline{\mathbb D})$ norms of $f_n'$ are uniformly bounded as well. Applying the Arzel\`a-Ascoli theorem to $f_n'$ we can assume that a subsequence converges uniformly on $\overline{\mathbb D}$ to an analytic disc, say $g$. Also, a subsequence of $(f_n)$ converges on $\overline{\mathbb D}$ to an analytic disc that we shall denote by $f$. Clearly $f'=g$ and a subsequence $f_n$ converges to $f$ in $\mathcal C^1(\overline{\mathbb D})$ topology. Note that $f$ is a (unique) complex geodesic in $D$, because
 it verifies the extremality conditions from \cite{L1}.
 \end{proof}

Actually,  something can still be said about complex geodesics in $D$ in the case when $\alpha=0$, that is when a strongly convex domain $D$ has $\mathcal C^2$-smooth boundary. Chirka, Coupet and Sukhov established in \cite[Corollary 1.5]{CCS} the $\mathcal C^{1-\epsilon}$ continuity up to the boundary for a holomorphic map $f$ from $\mathbb D$
with boundary values in a totally real $\mathcal C^1$ manifold $M$.  The manifold to be considered here
is, in the terminology of \cite{L3}, $\widehat{\partial D}:= \{(z, T_z^{\mathbb C} \partial D), z\in \partial D\} \subset\mathbb C^n\times \mathbb P^{n-1}$, while if we have a geodesic $\phi$, the map has to be $(\phi,\tilde \phi)$, see
\cite[proof of Lemma 3.1]{L3}.
Also, their result implies that one can control $\mathcal C^{1-\epsilon}$ norms of geodesic also under $\mathcal C^2$-perturbations of $\partial D$ (see the discussion following \cite[Corollary 1.5]{CCS}).

\section{Proof of Theorem~\ref{thm:main}}

We shall divide the proof into two cases, according to whether the direction generated by $z$ and $w$ is
``tangential'' or ``normal''. In order to give this a precise meaning, first recall that it is enough to 
prove our theorem for $z, w$ in a neighborhood $U$ of $\partial D$, which we will take small enough so that 
for any $z\in U\cap D$, there exists a unique point $\pi(z)\in \partial D$ such that $|z-\pi(z)|
= \min\{ |z-\zeta|, \zeta \in \mathbb C^n \setminus D\}$. We denote by $n_z$ the outer unit normal to 
$\partial D$ at $\pi(z)$, and write, for any vector $v$, $v_z:=\langle v,n_z\rangle n_z$ where $\langle \cdot, \cdot\rangle$ stands for the usual Hermitian inner product. 

We first deal with the (almost) ``tangential'' case.
\begin{lemma}\label{lem:eps}
Let $D$ be a strongly pseudoconvex domain with $\mathcal C^{2, \alpha}$ boun\-dary. There exists
$\epsilon_0>0$ such that if $0<\epsilon<\epsilon_0$, and if $z,w\in D$ satisfy $\delta(z)<\epsilon$, $|z-w|<\epsilon$ and $|(z-w)_z|$  $ < \epsilon |z-w|$, then $z$ and $w$ lie on a stationary map (thus on a complex geodesic) whose diameter is $o(1)$ as $\epsilon \to 0$. Moreover, $|(x-y)_x|/|x-y|$  $= o(1)$ for any points $x,y$ of this complex geodesic.
\end{lemma}

Part of this result was proven in \cite{BFW} under a $\mathcal C^3$ smoothness assumption (see also \cite{Kos}). Below we shall show how to deduce it in the $\mathcal C^{2,\alpha}$ setting directly from Lempert's work.

\begin{proof}
First make the additional assumption that $D$ is convex.
Since it is strongly pseudoconvex (bounded), it enjoys the
visibility property (\cite{BZ}, or use \cite{BB} and the fact that
the Gromov boundary of a Gromov hyperbolic space is always visible): given two distinct points $p,q \in \partial D$,
and sequences $p_n\to p$, $q_n\to q$, there exists a compact
set $K\subset D'$ such that for $n$ large enough, the real 
geodesic from $p_n$ to $q_n$ intersects $K$. It is easy to
see that this can be made uniform: as soon as $|p-q|\ge \delta>0$,
there is a $K_\delta \subset \subset D'$ such that geodesics from
$p_n$ to $q_n$ eventually intersect $K_\delta$.
In particular, if the diameter of the 
complex geodesic through $z$ and $w$ is big enough, and $z$ and $w$
are close to the boundary,
the complex geodesic needs to intersect a fixed compactum $K$,
and thus meets the boundary transversely (uniformly with respect to the point in $K$, by \cite[Proposition 14]{L1}). Therefore, for points on such a complex geodesic  $z,w$ which are close enough to each other and to the boundary, $|(z-w)_z|\geq C |z-w|,$ where $C>0$ is uniform.  This contradicts the assumptions with $\epsilon_0$ small enough, we would have a contradiction, so the assertion about the diameter is proved.

For general strongly pseudoconvex domains one can derive the assertion by exposing points as in \cite{DFW}: using a biholomorphism on a neighborhood of $\overline{D}$, 
we can assume that $D\subset \mathbb B$, $p\in \partial D \cap \partial \mathbb B$. Furthermore by the regularity assumptions $\partial D$ remains strongly convex in a neighborhood of $p$, so taking $\epsilon_0$ small enough we may assume $z_n,w_n, q \in U$, a neighborhood of $p$ such that $D\cap U$ is smooth strongly convex and $U\cap \partial D$ is close to $U\cap \partial \mathbb B$.
Consider the complex geodesic $f:\mathbb D\longrightarrow D\cap U$ obtained for the convex domain $D\cap U$ by the  argument above. This is a stationary map in the sense of Lempert \cite[Section 3, p. 562, points 1--3]{L3}, meaning in particular that there exists a $\mathcal C^{1/2}$ positive 
function $p(\zeta)$, $\zeta \in \partial \DD$, such that the function $\tilde f(\zeta):= \zeta p(\zeta) \overline{n_{f(\zeta)}}$ extends to a holomorphic function on $\DD$. Therefore it is also stationary as a map $f:\mathbb D\longrightarrow D$. 

Making $z,w$ even closer to $p$ as needed,  we may assume that $\overline{f(\DD)}\cap \partial (D\cap U)$ is contained in $\partial D$ and close enough to $p$ so that for any $z\in D$, $\re \langle z-f(\zeta), n_{f(\zeta)}\rangle <0$.
A winding number argument then shows that the equation 
$ \langle z-f(\zeta), \overline{\tilde f(\zeta) } \rangle=0$ admits
a unique solution $F(z)$ in $\DD$, which has to be holomorphic, and this uniqueness shows that $F(f(\zeta))=\zeta$ for any $\zeta \in \DD$ (in other words we have a left-inverse, see \cite{KZ}) and 
therefore the map $f$ is indeed a complex geodesic for $z,w$ with respect to $D$.

The last statement in the Lemma follows from \cite[Theorem~3]{KN}: if $x,y$ are in the range of a complex geodesic $\varphi$, then $|(x-y)_x|\lesssim |x-y| \diam (\varphi)$\footnote{As usual, $f\lesssim g$ means that $f\le Cg$ for some uniform constant $C>0$ (depending only on $D$); $f\simeq g$ is understood analogously}.
\end{proof}

\begin{corollary}\label{lem:uni}
Let $D$ be a domain in $\mathbb C^n$ and $p\in \partial D$ a strongly
pseudoconvex point. If $z,w\in D$ are close to $p$ and $|(z-w)_z|/|z-w|$ is sufficiently small, then there is a unique real geodesic for $z,w$ and it is induced by a complex one.
\end{corollary}

\begin{proof} 
Recall there is a neighborhood $U$ of $p$ so that there exists
a biholomorphism on a neighborhood of $\overline{D\cap U}$ making $D$
convex.

Let $\varphi$ be any complex geodesic passing through $z,w$. By Lemma
\ref{lem:eps}, the range of $\varphi$ is close to the boundary, its diameter is also small, so we
may assume $\varphi(\mathbb D)\subset U$. 
So after biholomorphism we can reduce ourselves
to stationary discs in a strictly convex domain $D'$.

Inside the complex geodesic $\varphi$ there can only be one real geodesic, that we shall denote  by $\gamma$, $\gamma(0)=z$, $\gamma(t_1) = w$. Suppose that there is another real geodesic between $z$ and $w$, say $\eta$. Then for some $t_0$ the point $\eta(t_0)$ must lie outside the range of $\varphi$ and
we can choose it so that $|(\eta(0)-\eta(t))_{\eta(0)}|/|\eta(0)-\eta(t)|$ 
is small for all $t\in[0,1]$. As in Lemma \ref{lem:eps} take a complex geodesic $\psi$ passing through $\eta(0)$ and $\eta(t_0)$;
it has to be close to $\eta$ when $t_0$ is chosen close to $\sup\{t: \eta(t)\notin \varphi(\mathbb D)$. 
Consider the one-parameter family of maps from the unit disc to $D'$ given by $(1-\theta)\varphi + \theta \psi$, then
they are extremal for $z,w$ and
for $\theta$ close to $0$,  \cite[Lemma 3.2]{KZ} shows that they must be proper into $D'$, a contradiction.
\end{proof}

\subsection{Theorem~\ref{thm:main} -- tangential case}

Let $z_n,w_n\in D$ be such that $z_n,w_n\to p\in \partial D$, and $|(z_n-w_n)_{z_n}|/|z_n-w_n|\to 0$. Let $\varphi_{z_n,w_n}$ be the unique complex geodesic that passes through $z_n$ and $w_n$. Let $\gamma_{z_n,w_n}$ be the unique real geodesic for $z_n,w_n$, induced by $\varphi_{z_n,w_n}$.

We shall follow the notation from Section~\ref{sec:FM}. 
For each $n$, take a point $p_n \in \overline{\varphi_{z_n,w_n}(\DD)}$ so that $p_n\to p$ and then apply the transformation $F_{p_n}$. Let $D_n:=F_{p_n}(D)$.
So we are reduced to the case where
$p=(1,0)\in \partial D_n$ lies in the range of the closed analytic disc $\overline{\varphi_{z_n,w_n}(\DD)}$ and $D_n$ near the point $p$ is given by the inequality
\begin{equation}\label{eq:D}
\{|z|^2 + o(|z-p|^{2+\alpha})<1\}.
\end{equation}
We now drop the subscript from the notation and write $D$ for $D_n$
and $z,w$ for $z_n,w_n$.  For $t\in (0,1)$, $D_t:=A_t^{-1}(D)$ as in \eqref{eq:Dt}.

Then for $t$ well chosen, 
close to $1$, $\psi_{z,w}:=A_t^{-1}\circ \varphi_{z,w}$ is a complex geodesic in $D_t$ that lies entirely in $\{\re z_1\geq 0\}$, with its boundary intersecting $\{\re z_1 =0\}$;  it is close in $\mathcal C^{1}$ topology on $\overline \DD$ to a geodesic in $\BB_n$, after a possible reparametrization. Its range intersects a fixed compact subset $K$ of $\mathbb B$, independent of $n$, and contains $(1,0)$ in its closure.

Complex geodesics in $\mathbb B_n$  are intersections of $\mathbb B_n$ with complex affine lines. Therefore, a simple argument shows that any
complex 
geodesic $\psi=(\psi_1, \tilde \psi)$ in $\mathbb B_n$ contained in $\{ \re z_1\geq 0\}$
that passes through $(1,0)$, and such that $\psi(0)\in K$, satisfies the following uniform estimates:
$$
\left| \frac{\tilde \psi(\mu)}{1+ t \psi_1(\mu)} - \frac{\tilde \psi (\lambda)}{1 + t \psi_1(\lambda)} \right| 
\simeq |\mu - \lambda|\simeq |\tilde \psi(\mu) - \tilde \psi(\lambda)|.
$$ 
Trivially, 
$$|\psi_1(\mu) - \psi_1(\lambda) |\lesssim |\lambda - \mu|,$$ again uniformly.
By $\mathcal C^1$ convergence argument, all these uniform estimates remain true for the geodesics $\psi_{z,w}$ (after proper reparametrizations). In particular, if $x=(x_1, x'), y=(y_1,  y')$ are two points lying in the range
of any $\psi_{z,w}$, then 
\begin{equation}\label{eq:xy} \left| \frac{y'}{1+ t y_1} - \frac{x'}{1+ tx_1}\right| \simeq |y' - x'| \simeq |y- x|.
\end{equation}

Let $x,y$ be in the range of  $\eta_{z,w}:=A_t^{-1}(\gamma_{z,w})$. Clearly,
$$A_t(x) - A_t(y) = \left( \frac{(1-t^2) (y_1 - x_1)}{(1 + t x_1)(1+ty_1)}, (1-t^2)^{1/2}( \frac{ y'}{1 + t y_1} - \frac{x' }{1 + tx_1}) \right).$$
In particular, by the estimates from \eqref{eq:xy}:
\begin{equation}\label{eq:A}|A_t(y) - A_t(x)|\simeq (1-t^2)^{1/2} |y- x|,
\end{equation}
where $\simeq$ is uniform.

As already mentioned, $\eta_{z,w}$ are real geodesics, induced by 
$\psi_{z,w}$.
The geodesics $\psi_{z,w}$ intersect a fixed compact set $K$. 
By Proposition \ref{prop:c1conv} and its proof,
they are uniformly $\mathcal C^{1+\alpha}$-smooth on the closed disc in the sense that their $\mathcal C^{1+\alpha}$-norms (so $\mathcal C^1$ as well) are uniformly bounded, which implies the assertion of Theorem~\ref{thm:main} for $\eta_{z,w}$. 
From this and \eqref{eq:A} one can deduce the assertion for $\gamma_{z,w}=A_t(\eta_{z,w})$.


\subsection{Theorem~\ref{thm:main} -- non-tangential case: preparatory results}

A visibility
lemma for the class of ``Goldilocks'' domains (which includes
finite type domains) was given by Bharali-Zimmer \cite{BZ}. 
The following visibility lemma is essentially due to Bracci-Fornaess-Wold \cite{BFW},
but we need a slightly more general form:
\begin{lemma}\label{lem:1}
Let $\alpha>-1$ and $D_n$ be a family of domains $D_n \subset \mathbb B$ such that $D_n\cap \{\re z_1\geq \alpha\}$ converges to $\mathbb B\cap \{\re z_1\geq \alpha\}$ in $\mathcal C^2$ topology. Let $z_n, w_n\in D_n$ converge to $p$ and $q$ respectively, where $p,q\in \partial \mathbb B \cap \{\re z_1 >\alpha\}$, $p\neq q$. Then there is $K\subset \subset \mathbb B$ such that any real geodesic $\sigma_n$ in $D_n$ for $z_n,$ $w_n$ intersects $K$.
\end{lemma}

If instead of a sequence of domains $D_n$ a fixed one is considered, the result was proven in \cite{BFW} (see also \cite{BZ}). A glimpse at their reasoning ensures us that estimates achieved there depend just on the normal curvature of the domain.

\begin{proof}

We can reparametrize $\sigma_n:[a_n,b_n]\to D_n$ by arc-length so that $\delta_n(\sigma_n (0))\leq \delta_n (\sigma_n(t))$, where $\delta_n(x)=\dist (x, \partial D_n)$. Since $D_n$ are uniformly bounded, $|\sigma_n'(t)|\leq L$ for some uniform constant $L$.

On the other hand, for $n$ big enough the normal curvature of $\partial D_n$ restricted to $\{\re z_1 \geq \alpha +\epsilon\}$, where $\epsilon$ is small enough, is close to the one of $\mathbb B$. In particular, all $D_n$ satisfy an exterior ball condition with a uniform constant. From this we deduce the existence of $C>0$ such that \begin{equation}
\label{eq:krbig}
    k_{D_n}(x, v)\geq C \frac{||v||}{\sqrt{\delta_n(x)}}
\end{equation}
for $x\in D_n\cap \{\re z_1 \geq \alpha + \epsilon\}$.

The same curvature-type argument shows that $\partial D_n$ restricted to $\{\re z_1 \geq \alpha +\epsilon\}$ satisfies a uniform interior ball condition. Therefore, there is a uniform $k\in \mathbb N$ so that any two points in $D_n$ can be connected with a chain of $k$ balls of a fixed radius. From this, we gain an estimate 
$$K_{D_n}(z,w)\leq C + \frac 12 \log \frac{1}{\delta_n(z)} + \frac12 \log \frac{1}{\delta_n(w)}, $$ with a uniform $C>0$. 

Now it is enough to notice that a curve that remains too close to the boundary at all times
has length bigger than the above estimate for the Kobayashi distance because of \eqref{eq:krbig}, or to
repeat an argument from \cite[Proof of Proposition 2.2]{BFW} or \cite[Section 5]{BZ}.
\end{proof}

\begin{lemma}\label{lem:claim}
Suppose that $D_{t_n}$ converges to $\mathbb B$ as in A1) and A2) when $t_n\to 1$. Let $z_n,w_n\in D_{t_n}$ be such that $z_n\to (1,0)$ and $w_n\to (-1,0)$. Let $\varphi_n$ be a real geodesic for $k_{D_n}$ joining $z_n$ and $w_n$, parametrized so that for some $\alpha_n,
\beta_n >0$, $\varphi_n(-\alpha_n)=w_n$, $\varphi_n(\beta_n)=z_n$,
and $\re \varphi_n(0)_1=0$.
Then inside any compact subset $K\subset \mathbb B^N$,
the geodesics $\varphi_n$ converge  in the $\mathcal C^1$ topology to a geodesic line in $\mathbb B$ between $(-1,0)$ and $(1,0)$, and for any $\epsilon, \delta_0>0$, there exists
$n(\epsilon,K,\delta_0)$ such that for any $z',w' \in K\cap \varphi_n$ 
with $|z'-w'|\ge \delta_0$, then for $n\ge n(\epsilon,K,\delta_0)$,
$l(\varphi_n|_{[z',w']})\le (1+\epsilon)|z'-w'|$.

The same holds if we change the assumption on $z_n$ to the following: $z_n\in (0,1)\times \{0\}^{N-1}$ (possibly replacing the geodesic line by a ray).
\end{lemma}

\begin{proof}
Take $\tilde \beta< \beta<1$ close to $1$. 
Take $s_n:=\max\{ t: \re \varphi_n(t)_1= - \beta\} $,
$\tilde s_n:=\min\{ t: \re \varphi_n(t)_1= -\tilde \beta\} $,
$\tilde t_n:=\max\{ t: \re \varphi_n(t)_1= \tilde \beta\} $,
$t_n:=\min\{ t: \re \varphi_n(t)_1=  \beta\} $,
and corresponding points 
$x_n=\varphi_n(s_n)$, $\tilde x_n=\varphi_n(\tilde s_n)$, 
$\tilde y_n=\varphi_n(\tilde t_n)$ and $y_n=\varphi_n(t_n)$. It follows from the visibility lemmas applied to $x_n, \tilde x_n$ and to $y_n, \tilde y_n$ that there are compact sets $K_1\subset \{-\beta< \re z_1 < - \tilde \beta\}$ and $K_2 \subset \{\tilde \beta < \re z_1 < \beta\}$ that intersect the ranges of the geodesics $\varphi_n$.

Since the $\varphi_n$ form an equicontinuous family, we get that a subsequence of $\varphi_n$ converges to a real geodesic
$\varphi$ in the ball that intersects $K_1$ and $K_2$
and passes through $0$. 
Letting $\tilde\beta$ tend to $1$, we get the  convergence on all
compact sets. The statement about lengths comes from the fact that the 
real geodesic for the ball between $z'$ and $w'$ is close to a line.

The second assertion can be deduced in a similar way.
\end{proof}

\subsection{Proof by contradiction: Claims}

Suppose that the estimate in Theorem \ref{thm:main} fails. 
Then there are sequences of points $z_n,w_n\in D$,
joined by a geodesic $\gamma_n = \gamma_{z_n, w_n}$ 
such that 
\begin{equation}
\label{toolong}
l(\gamma_n) \geq a_n |z_n - w_n|, \mbox{ where }a_n\to \infty.
\end{equation}

Our goal is to obtain Claim 1, and then to prove Claim 2 which
contradicts the previous one.
 
Since the ratio of the Euclidean and Kobayashi-Royden infinitesimal metrics are bounded
above and below on any compactum, it is enough to consider the case at least one of the points $z_n,w_n$
escapes from any compact set, so passing to a subsequence
and exchanging $z_n$ and $w_n$ if needed, we may
assume $z_n\to p\in \partial D$. 

Then we may reduce ourselves to the case where $w_n$
converges to the same boundary point. Indeed, if the distance between $z_n$ and $w_n$ did not converge to 0, then 
by \eqref{toolong} the lengths of the curves
$\gamma_n$ would go to infinity and switching $z_n$ and $w_n$, if necessary, we could pick points $w'_n \in \gamma_n$, 
such that $|w'_n-z_n|\to 0$ while $l(\gamma_n|_{[w_n',z_n]})$ remains big.

\bigskip

{\bf Claim 0.} It is enough to focus on the case when $|\left(\gamma_n(t) - \gamma_n(s)\right)_{\gamma_n(t)}|\geq C |\gamma_n(t) - \gamma_n(s)|$, 
for all $s, t$, for some uniform constant $C$, where normal length is taken with respect to $p\in \partial D$. 

\begin{proof}[Proof of the claim] Actually, if the opposite inequality were true for some $t,s$, where $C$ is small enough, we would deduce that $\gamma_n$ is contained in a complex geodesic, according to Corollary~\ref{lem:uni}. From Lemma~\ref{lem:eps} we would get that this is actually covered by the tangential case. 
\end{proof}

It follows from Claim 0 that it suffices to measure the length of $\gamma_n$ projected to the normal direction. Denote this {\it normal} length by $l_N$.

\bigskip

{\bf Claim 1}. There exist a sequence of domains $D_n$, points $\zeta_n$, $\nu_n\in D_n$ and real geodesics $\eta_n$ in $D_n$ joining them such that $\zeta_n=\eta_n(t_n)$, $\nu_n=\eta_n(-s_n)$, and
\begin{enumerate}
\item For any $\beta>-1$, the domains $D_n$ converge to $\mathbb B$ in $\mathcal C^{2,\alpha}$ topology when restricted to $\{\re z_1\geq \beta\}$;
\item  $(1,0)\in \partial D_n$ and the defining functions of $D_n$ near $(1,0)$ are of the form \eqref{eq:D};
\item $\zeta_n\in (0,1)\times \{0\}^{N-1}$, $\zeta_n\to (1,0)$, and $\nu_n\to (-1,0)$, and setting $\xi_n=\eta_n(0)$,
$\re (\xi_n)_1=0$, $\xi_n\to (0, x)$ as $n\to \infty$;
\item $l_N(\eta_n|_{[0,t_n]})\to \infty$, as $n\to \infty$.
\end{enumerate}

\bigskip

We shall prove Claim 1 in several steps. 

The following simple observation will be used several times, so let us state it as a separate result.
Recall that $m_t$ is defined in \eqref{eq:aut}.

\begin{lemma}\label{lem:mt} 1) Let $\delta:I\to \DD$ be a curve. Let $t\in (0,1)$  be such that $\tilde \delta:=m_t^{-1}\circ \delta $ is contained in $\{\re \lambda \geq \beta\}$, $\beta>-1$. Then 
$$| \delta(s_1) - \delta(s_2)|\simeq (1-t^2) |\tilde \delta (s_1) - \tilde \delta(s_2)|,\quad s_1,s_2\in I.$$ 
Consequently, 
$$l(\delta)\simeq (1-t^2) l(\tilde \delta).$$ 
The estimates above depend only on $\beta$.

2) Suppose additionally
\begin{equation}\label{eq:s}
\left| \tilde\delta(s_2) - \tilde\delta(s_1)\right| \le 
l(\tilde \delta |_{[s_1,s_2]} )\leq (1+ \epsilon) \re(\tilde\delta(s_2) - \tilde\delta(s_1)),
\end{equation} 
where $s_1, s_2\in I$.
Assume also that $|\im \tilde \delta(s)|<\epsilon$, $s\in I$,  where $\epsilon>0$ is small enough. Then $$l(\delta|_{[s_1,s_2]}) \leq  C_{\beta} (1+ \epsilon) \re (\delta(s_2) - \delta(s_1)).$$ 
\end{lemma}

\begin{proof}
If $\re x_1, \re y_1\geq -\beta$, then $|1+t x_1|$ and $|1+t y_1|$ are between $1-\beta$ and 2, so by direct computations $|m_t(x_1) -m_t(y_1)|\simeq  (1-t^2)|y_1-x_1|$.

The second part is technical, as well. The first part
and assumption \eqref{eq:s} imply that $l(\delta|_{[s_1, s_2]}) \leq c_\beta (1+\epsilon) (1-t^2) \re (\tilde\delta (s_2) - \tilde \delta (s_1)).$

To get the assertion, compute
\[ 
\re (\delta(s_2) - \delta(s_1)) = (1-t^2) 
\re \left(     
\frac{\tilde\delta(s_2) - \tilde\delta(s_1)}{(1+t\tilde\delta(s_1)) (1+t\tilde\delta(s_2))} \right)
\]
Since $|\im \tilde \delta(s)|<\epsilon$, the argument of the 
denominator is bounded by $C\epsilon$, and by \eqref{eq:s}, so is
the argument of the numerator. So we get
\[
\re (\delta(s_2) - \delta(s_1)) \geq
(1-t^2)(1-C\epsilon)\re \left( \tilde\delta(s_2) - \tilde\delta(s_1)\right),
\]
which finishes the proof.
\end{proof}

Let us come back to the proof of Claim 1. Let $p_n\in \partial D$ minimize the distance from $z_n$ to $\partial D$. Clearly $p_n\to p$.
Using a Fridman-Ma transformation, we can assume that $p_n=(1,0)$ and that $\partial D$ is near $(1,0)$ of the form \eqref{eq:D}. Let $A_t$ be given by \eqref{eq:aut}. 

Choose $t_n$  such that $A^{-1}_{t_n}(\gamma_n) \subset \{\re z_1\ge 0\}$ and 
$A^{-1}_{t_n}(\gamma_n) \cap \{\re z_1=0\} \neq \emptyset$.

\bigskip

{Step 1}. $t_n\to 1$ as $n\to\infty$. 

\begin{proof} 
This comes from the fact that a 
Gromov hyperbolic metric space $X$ (with boundary) does not have geodesic loops, i.e. isometric maps $\gamma$ from $\mathbb R$ 
such that $\lim_{t\to\pm \infty} \gamma(t)=p \in \partial_G X$, where $\partial_G X$ is the Gromov 
boundary of $X$ and the convergence is in the sense defined for that boundary. Recall that a bounded strongly 
pseudoconvex domain in $\mathbb C^n$ endowed with the Kobayashi distance is Gromov hyperbolic and
its Gromov boundary coincides with the Euclidean boundary \cite{BB}.

Suppose that some subsequence, again denoted by $t_n$, remains
bounded away from $1$. We may restrict attention to a neighborhood of
 $p$ and thus assume it
 is  a point of global strong convexity. The geodesics $(\gamma_n)$ pass through $z_n$ and points on the surface $A_{t_n}\left(\{\re z_1=0\}\right)$, which are far away from each other. Thus the visibility applies and allows us to choose $K \subset \subset \mathbb B$ such that every $\gamma_n$ intersects it. 

Choose a sequence $s_n\to 1$ slowly enough so that $A_{s_n}^{-1}(z_n)$ and $A_{s_n}^{-1}(w_n)$ still tend  to $p$,
so the starting and
ending points of $A_{s_n}^{-1}(\gamma_n)$ are both close to $(1,0)$. But $A_{s_n}^{-1}(K)$ converges to the point $(-1,0)$. Whence $A_{s_n}^{-1}(\gamma_n)$ is convergent to a nontrivial geodesic line in $\mathbb B$ with both extremities tending to the same boundary point, a contradiction with the 
non existence of geodesic loops.
\end{proof}

Let $\tilde z_n = A_{t_n}^{-1}(z_n)$,  $\tilde w_n=A_{t_n}^{-1}(w_n)$, and $\tilde \gamma_n := A^{-1}_{t_n} (\gamma_n)$.

\bigskip

{Step 2}. The following uniform estimates hold:
$$l_N(\gamma_n)\simeq (1-t_n^2) l_N(\tilde \gamma_n),
$$ and $$|z_n - w_n|_{z_n} \simeq (1-t_n^2) |\tilde z_n - \tilde w_n|$$ as $n \to \infty$. Consequently, 
\begin{equation}
\label{toolong2}
    l( \tilde \gamma_n)/|\tilde z_n - \tilde w_n |\simeq l (\gamma_n)/|z_n - w_n|\geq a_n.
\end{equation}

\begin{proof}
This is a direct consequence of Lemma~\ref{lem:mt}.
\end{proof}

{Step 3}. There is no $K\subset \subset \mathbb B$ such that both $\tilde z_n$ and $\tilde w_n$ lie in $K$ for any $n\in \mathbb N$ large enough.

\begin{proof}
If there was, there would be a constant $C_K>1$ such that
$C_K^{-1} k_{D_n}(z;v) \le \|v\| \le C_K k_{D_n}(z;v)$ for $z\in K$,
and therefore $C'_K>1$ such that
$(C'_K)^{-1} |\tilde z_n - \tilde w_n | \le  l( \tilde \gamma_n)
\le C'_K |\tilde z_n - \tilde w_n |$, which contradicts 
\eqref{toolong2}.
\end{proof}

{Step 4}. $\liminf_{n\to \infty}|\tilde z_n- \tilde w_n|>0$.

\begin{proof} If this statement were false, then (possibly passing to a subsequence) $\tilde z_n,\tilde w_n$ would converge to $(1,0)$, according to Step 3 and the fact that $\tilde z_n\in (0,1)\times \{0\}^{N-1}$. Since the geodesics $\tilde \gamma_n$ touch $\{\re z_1=0\}$ (which is far from $(1,0)$), the argument from Step 1 provides us with a contradiction.
\end{proof}

{Step 5}. $l_N(\tilde \gamma_n)\to \infty$ as $n\to \infty$. 

\begin{proof}
A direct consequence of Steps 2 and 4.
\end{proof}

\begin{proof}[End of proof of Claim 1]
From Step 4 and Lemma~\ref{lem:1} we deduce that $\tilde \gamma_n$ intersects a fixed compact subset $K_0$ of $\mathbb B$. For each $n$, choose  $x_n \in K_0 \cap \tilde \gamma_n$. Step 2 ensures that the lengths of curves from $x_n$ to $\tilde z_n$ or from $x_n$ to $\tilde w_n$ converge to infinity. If the first possibility holds, then by the arguments from Step 3 $\tilde z_n$ goes to the boundary, and we put $\zeta'_n=\tilde z_n$ and $\nu'_n=x_n$. 

If the second one holds, we apply Fridman-Ma construction with respect to points $q_n\in \partial D_n$ from Lemma~\ref{rem}. It follows  from the Fridman-Ma construction that it preserves the fact that
lengths go to infinity or not. Then define $\zeta'_n$ to be (the new) $\tilde w_n$, and $\nu'_n$ to be (the new) $x_n$. 

Now we need to redefine $\zeta'_n$  and $\nu'_n$ a bit so that $\nu_n$ escapes to $(-1,0)$. It will be straightforward: the $\nu'_n$ remain within a fixed compact set that will be pushed towards $(-1,0)$ under appropriate 
automorphisms of the ball.

We set $\eta_n:=A_{s_n}^{-1}(\tilde \gamma_n)$, $\zeta_n:=A_{s_n}^{-1}(\zeta'_n)$, 
and $\nu_n:=A_{s_n}^{-1}(\nu'_n)$ where $(s_n)_n\subset (0,1)$ is to be chosen
increasing to $1$, so that $\nu_n$ tends to $(-1,0)$. Now let $\xi_n$ be a point on $\eta_n \cap \{\re z_1 =0\}$. We can pick $(s_n)_n$ converging to $1$ slowly enough so that the lengths of the curves $\eta_n$ from $\xi_n$ to $\zeta_n$ still tend to infinity.

The assertion on normal lengths is clear (compare Claim 0 and Step 2). 
\end{proof}

\medskip

{\bf Claim 2.} Let $D_n$, $\eta_n$ be as in Claim 1 and satisfy properties (1), (2), and (3). Then 
\begin{equation}\label{eq:o1} \frac{l_N(\eta_n|_{[0,t_n]})}{| \eta_n(0)- \eta_n(t_n)|_{\eta_n(0)}} =O(1),
\end{equation}
as $n\to \infty$.

\medskip
\begin{proof}[Proof of Claim 2]

Let $A^{-k} := A_{1/2}^{-1} \circ \cdots \circ A_{1/2}^{-1}$ and $m^{-k} := m_{1/2}^{-1} \circ \cdots \circ m_{1/2}^{-1},$ where the automorphism $A_{1/2}$ and M\"obius map are defined in \eqref{eq:aut}.


First note that composing with $A^{-k}$ makes domains $D_n$ closer to $\mathbb B$ when $k\to \infty$. For the simplicity of notation write $\delta= \eta_n^1$ (the first complex coordinate of $\eta_n$). 

Fix $\epsilon>0$. Since $\zeta_n\in (0,1)\times \{0\}^{N-1}$, 
we have $m^{-k}((\zeta_n)_1)>0$ for $k\le k(n)$, and
since $ m^{-k}(\nu_n)$ is closer to $(-1,0)$ than $\nu_n$ for $k\ge0$,
it follows from the second assertion of
Lemma~\ref{lem:claim} applied to $m^{-k}(\eta_n) $
that if $n$ is big enough, then $|\im(m^{-k}(\delta(s)))|<\epsilon$ for $s$ and $k\geq 0$ such that $\re(m^{-k}(\delta(s)))=0$.
Fix $-1<\beta_1<\beta_2<0$ so that $m^{-1}(i [-\epsilon,\epsilon])\subset \{\beta_1<\re z_1<\beta_2\}.$ 

By the Denjoy-Wolff theorem $m^{-k}((\zeta_n)_1)\to -1$ as $k\to\infty$. Therefore we can pick a finite number of points $0=t_0<s_1< \ldots< s_k<t_n \leq s_{k+1}$ that satisfy the following equalities $\re m^{-1}(\delta(s_1))=0$, $\re m^{-2}(\delta(s_2))=0, \ldots , \re m^{-(k+1)}(\delta(s_{k+1}))=0$. 

It follows from our construction that $\beta_1<\re(m^{-(j+1)}(\delta(s_j)))<\beta_2$ while $\re(m^{-(j+1)}(s_{j+1}))=0$.
by Lemma~\ref{lem:claim},  we have a bound on the imaginary part of $m^{-(j+1)}\circ \delta$ on $[s_j, s_{j+1}]$. So \eqref{eq:s} is satisfied, and  Lemma~\ref{lem:mt} applied to $\tilde\delta|_{[s_1, s_2]}:= m^{-(j+1)}\circ \delta|_{[s_j, s_{j+1}]}$ and $m^{-1}_t:=m^{-(j+1)}$ provides us with the  estimate $$l(\delta_{[s_{j}, s_{j+1}]})\leq C(1 + \epsilon) \re(\delta(s_{j+1}) - \delta(s_j)),$$ where the constant $C$ is uniform. From this, we trivially get the assertion.
\end{proof}

\begin{proof}[Proof of Theorem~\ref{thm:main}]
Claims 1 and 2 contradict each other.
\end{proof}

{\bf Acknowledgments.} The authors wish to thank the anonymous referees for their multiple observations which greatly improved the exposition of this paper.

\end{document}